\documentclass[9pt,a4paper,english]{extarticle}

\pdfoutput=1
\usepackage{abstract}
\usepackage{amssymb}
\usepackage{amsmath}
\usepackage{amsthm}
\usepackage[boxed, vlined, linesnumbered, resetcount]{algorithm2e} 
\usepackage{bm}
\usepackage{bbm}
\usepackage{datenumber}
\usepackage{datetime}
\usepackage{dsfont}
\usepackage{float}
\usepackage{listings}
\usepackage{mathtools}
\usepackage{multirow}
\usepackage[numbers, sort&compress]{natbib}
\usepackage[super]{nth}
\usepackage{physics}
\usepackage[scr,scaled=1.1]{rsfso}
\usepackage{subfigure}
\usepackage[binary-units=true]{siunitx}
\usepackage{todonotes}
\usepackage{hyperref}
\usepackage{caption}
\usepackage[top=20mm,bottom=20mm,left=20mm, right=20mm]{geometry}

\newfloat{lstfloat}{htbp}{lop} 

\SetAlCapSty{} 
\SetAlCapSkip{1em} 
\SetAlCapNameFnt{}
\SetAlCapFnt{}
\SetAlgoCaptionSeparator{:}

\lstdefinestyle{C}{
    language=C,
    basicstyle=\small\ttfamily,
    keywordstyle=\small\ttfamily,
    morekeywords={omp,simd,reduction,simdlen,declare,inline,bool,restrict,half},
    otherkeywords={\#pragma,\_\_fp16},
    frame = single,
    captionpos=b,
    abovecaptionskip=1em,
}

\newcommand{\indicatorfn}{\mathds{1}}

\newtheorem{theorem}{Theorem}[section]

\newtheorem{lemma}[theorem]{Lemma}
\newtheorem{assumption}{Assumption}[section]
\newtheorem{model}{Model}[section]

\title{\Huge Rounding error using low precision approximate random variables}

\author{
\href{mailto:mike.giles@maths.ox.ac.uk}{Michael Giles}%
\thanks{\href{mailto:mike.giles@maths.ox.ac.uk}%
{\texttt{mike.giles@maths.ox.ac.uk}}} 
\and 
\href{mailto:oliver.sheridan-methven@hotmail.co.uk}{Oliver Sheridan-Methven}%
\thanks{\href{mailto:oliver.sheridan-methven@hotmail.co.uk}%
{\texttt{oliver.sheridan-methven@hotmail.co.uk}}}
}

\date{
Mathematical Institute, Oxford University, UK\\[1em]
\datedayname\ \nth{\number\day} \monthname\  \number\year}

\begin{document}

\maketitle

\begin{abstract}
For numerical approximations to stochastic differential equations using the Euler-Maruyama scheme, we propose incorporating approximate random variables computed using low precisions, such as single and half precision. We propose and justify a model for the rounding error incurred, and produce an average case error bound for two and four way differences, appropriate for regular and nested multilevel Monte Carlo estimations. By considering the variance structure of  multilevel Monte Carlo correction terms in various precisions with and without a Kahan compensated summation, we compute the potential speed ups offered from the various precisions. We find single precision offers the potential for approximate speed improvements by a factor of 7 across a wide span of discretisation levels. Half precision offers comparable improvements for several levels of coarse simulations, and even offers improvements by a factor of 10--12 for the very coarsest few levels. 

\begin{description}
\item[Keywords:] approximations, random variables, inverse cumulative distribution functions, random number generation, finite precision, half precision, floating point, rounding error, multilevel Monte Carlo, the Euler-Maruyama scheme, the Milstein scheme, Kahan compensated summation, and high performance computing.
\item[MSC subject classification:] 	65G50, 65C10, 41A10, 65C05, 65Y20, 60H35, 65B10, 65L70, 34M30, 97N20, and 65C30.
\end{description}
\end{abstract}

\section{Introduction}
\label{sec:introduction}

Rounding error has long been a source of scientific interest (and frustration). For a large fraction of the scientific community, the effects of rounding error are negligible and of little or no consequence. However, for an appreciable portion of the community, especially those pushing computer hardwares and numerical algorithms to the fastest speeds achievable, rounding error can present a considerable hurdle to the achievable fidelity and speed.

The example \textit{par excellence} of rounding error in scientific computing is in summation operations, frequent in linear algebra applications and to a lesser extent some statistical applications. Calculating the scalar product between two vectors, and thus also calculating vector and matrix multiplications, involves (among other things) summing a list of numbers. Being able to accurate sum a list of numbers has long been under the attention of mathematicians and computer scientists \citep{knuth2014art,kahan1965further,higham1993accuracy,trefethen1997numerical}, and its importance in scientific computing cannot be understated. When the numbers being summed are ill conditioned, the impact of rounding error grows with the problem's size, and can quickly nullify even the simplest of calculations, and thus high accuracy summation algorithms are frequently required. Outside of linear algebra, the accuracy of gradient and sensitivity estimates from finite difference methods and numerical differentiation is capped by the maximum available precision due to rounding error. Lastly, for the numerical solution of differential equations by simulation methods (deterministic or stochastic), iterative methods such as the Euler scheme incur rounding errors which get worse as the simulation's discretisation becomes finer. 

To combat the effects of rounding error, there are two particularly common approaches. The first is to try and bypass the issue by simply working in a higher precision, typically at the cost of computational speed. Historically this has motivated the introduction of double precision, extended double precision, and even quadruple precision data types. Similarly, several software libraries offer arbitrary levels of precision, such as: the mpmath \citep{mpmath} Python library, the GNU
multiple precision (GMP) arithmetic C/C++ library \citep{granlund2012gmp}, and the GNU multiple precision binary floating point with correct rounding (MPFR) C library \citep{fousse2007MPFR}. In a similar vein, hardware providers have also focussed efforts on improved floating point accuracy and reproducibility  \citep{burgess2018high}.

The second approach to reduce the influence of rounding error is to try and compensate and correct against it. For summations, the best known approach is the Kahan compensated summation \citep{kahan1965further}, although other compensation procedures have also been introduced and well explored \citep{moller1965quasi,knuth2014art,dekker1971floating,neumaier1974rounding,babuska1968numerical,klein2006generalised,linz1970floating,ogita2005accurate}. These proceed by inferring an estimate for the rounding error introduced at each stage of the summation, and then discount for this in the subsequent summations, thus compensating for the rounding error. 

There is a third means of circumventing rounding error, which is more mathematical in its nature, which looks to extrapolate accurate answers from less accurate approximations. The best example of this is Richardson extrapolation \citep{richardson1927viii,marchuk2012difference}. Without this technique, approximation schemes would need to go to such fine discretisations that rounding error would be significant, whereas using Richardson extrapolation is one possible technique at avoiding encountering rounding error. However, the use of Richardson extrapolation is very problem specific, and whilst it is a very powerful mathematical technique, it does not readily present itself as a general purpose computational tool for avoiding rounding error in most circumstances.  

Both high precision libraries and rounding error compensation schemes have historically been reserved for specialised applications seeking extraordinarily high accuracy, and closer to the edges of most scientific computing applications. However, in more recent years there has been an increased demand for ever lower precisions and data types, such as the IEEE half precision float \citep{ieee2008ieee}, and the more recent ``brain float'' \citep{burgess2019bfloat16,kalamkar2019study}. The large driving force behind these is the increasingly popular demand in machine learning applications, where the underlying data is very noisy and imprecise, and smaller data types are preferable for faster accessing, storage, and computation on the latest CPU, GPU, and TPU hardwares. Furthermore, with the greater desire for increased parallelisation on vector hardware and reduced precision calculations, lower precision data types are gaining considerable momentum and traction. 

In order to understand the nature of the nett rounding error arising in calculations, there have been two fronts of development. The first has been defining the precise rounding modes and data types used in scientific calculations. To ensure floating point calculations were standardised, the famous IEEE 754 standard for floating point arithmetic was introduced \citep{ieee1985ieee}, and is now the industry standard. This entails addition, subtraction, multiplication, division, and square roots all producing exact results with respect to the appropriate rounding mode \citep[page~15]{tucker2011validated}. Similarly, the rounding modes a computer uses to round floating point values are standardised, with ``round to nearest even'' typically being the default mode. Of course, while floating point arithmetic may be well defined, there are several difficulties and nuances, as discussed by \citet{goldberg1991every}.

The second front has been with the mathematical modelling of rounding errors. While the IEEE 754 standard specifies the hardware's behaviour, this does not readily give insight into the behaviour of the emergent nett rounding error. Describing the nett effect that results during calculations has received much mathematical attention \citep{higham2002accuracy,wilkinson1961error,wilkinson1974numerical,wilkinson1986error,hull1966tests}, and one of the best overviews is by \citet{higham2002accuracy}, who analyses the standard model for deterministic rounding error \citep[2.2, (2.4)]{higham2002accuracy}. Furthermore, in recent years there has been a piqued interest in stochastic rounding modes and associated error models, with prominent recent work by \citet{higham2019new} and \citet{ipsen2019probabilistic} performing probabilistic error analyses, which frequently give tighter and more realistic error bounds than the worst case deterministic scenarios.  

With this wealth of attention from academia and industry, the work we present produces a model for the rounding error incurred during the numerical simulation of stochastic differential equations. Typical treatments of numerical methods for such stochastic differential equations assume no rounding error occurs or is otherwise negligible \citep[9.3, page~316]{kloeden1999numerical} \citep{glasserman2013monte}. The earliest work to compensate for rounding error in simulations for ordinary differential equations appears to be by \citet{vitasek1969numerical}. In the setting of stochastic differential equations, the most relevant works are by \citet{arciniega2003rounding} and \citet{omland2016mixed}. \citet{arciniega2003rounding} present an \textit{ad hoc} statistically motivated model for the rounding error which occurs in the Euler-Maruyama scheme, giving an average case bound for the overall rounding error. \citet{omland2016mixed} takes a more rigorous approach, closer to a first principles model, starting with the floating point rounding modes and standard error model by \citet{higham2002accuracy}, and produces a worst case bound for the error in the Euler-Maruyama scheme \citep[theorem~4.8]{omland2016mixed}. 

The contribution of this work will be to present a heuristic model for the rounding error in a similar manner to \citet{arciniega2003rounding}. However, our model will be much more rigorously justified by a detailed inspection of the dominant rounding errors anticipated in the Euler-Maruyama scheme. Furthermore, we will show that there are two primary sources of error which contribute to the nett error. The first is a zero mean process, similar to that described by \citet{arciniega2003rounding}. The second is a possibly non zero mean systematic error term omitted by \citet{arciniega2003rounding}. This second process is of a much smaller size than the first, but due to its non zero mean nature, its nett contribution will grow at the same rate at that arising from the zero mean process. The significance of this new model is that it quantifies the permissible systematic rounding errors in the Euler-Maruyama scheme. Furthermore, our model is amenable to incorporation within the nested multilevel Monte Carlo scheme utilising approximate random variables developed by \citeauthor{giles2020approximate} \citep{giles2020approximate,sheridan2020approximate_inverse,sheridan2020approximate_random,sheridan2020nested}. Although work has been done by \citet{brugger2014mixed} and \citet{omland2015exploiting} on constructing multilevel Monte Carlo schemes in the presence of rounding error, our model directly facilitates a treatment jointly allowing for approximate random variables and low precision calculations, correctly handled by a nested multilevel Monte Carlo scheme. A secondary contribution of this work will also be to demonstrate the applicability of a Kahan compensated summation within the Euler-Maruyama scheme, an extension of the similar idea by \citet{vitasek1969numerical} in the setting of ordinary differential equations. 

Section~\ref{sec:numerical_solutions_to_stochastic_differential_equations} overviews the numerical solution of stochastic differential equations, providing the primary context and setting of our work, presenting our model for the leading order error process arising in the Euler-Maruyama scheme. Section~\ref{sec:multilevel_monte_carlo} will showcase how our model can be incorporated into a multilevel Monte Carlo framework, demonstrating practical applications of the model and highlighting the savings that can be expected using low precisions. Section~\ref{sec:conclusions} presents the conclusions from this work. 

\section{Numerical solutions to stochastic differential equations}
\label{sec:numerical_solutions_to_stochastic_differential_equations}

There are various settings appropriate for analysing the effects of rounding error, and the numerical solutions of stochastic differential equations is one particularly important setting. Frequently the terminal solution $ X_T $ of the stochastic differential equation $ \dd{X_t} = a(t, X_t) \dd{t} + b(t, X_t)\dd{W_t} $ needs to be approximated for given drift and diffusion processes $ a $ and $ b $. To achieve this, whole path approximations $ \widehat{X}_t \approx X_t $ for $ t \in [0, T] $ are produced, where the most popular methods are the Euler-Maruyama and Milstein schemes. For a thorough  detailing see \citet{kloeden1999numerical} and \citet{glasserman2013monte}. The approximations simulate the process over $ N $ time steps of size $ \Delta t \equiv \delta \coloneqq \tfrac{T}{N} $, where the update at the $ n $-th iteration at time $ t_n \coloneqq n\delta $ requires a Wiener process increment $ \Delta W $. The usual numerical schemes use a standard Gaussian random variable $ Z_n $ to simulate from this process, where $ \Delta W_n \coloneqq \sqrt{\delta} Z_n $. 

To ensure the stochastic process has a unique strong solution and the Euler-Maruyama scheme converges, we assume the standard assumptions from \citet[4.5]{kloeden1999numerical}, which are that: $ a $ and $ b $ are jointly Lebesgue measurable, spatially Lipschitz continuous, have linear spatial growth, $ \tfrac{1}{2} $-H\"{o}lder temporal continuity with linear spatial growth, and that $ X $ has a measurable initial condition.

\subsection{Approximate random variables}
\label{sec:approximate_random_variables}

Unfortunately, sampling from the Gaussian distribution is expensive, and so there have been several approaches to bypass this cost. Most of these look to substitute the exact Gaussian increment $ Z_n $ with another random variable $ \widetilde{Z}_n $ with similar statistics. For clarity and consistency with \citeauthor{giles2020approximate} \citep{giles2020approximate,sheridan2020approximate_inverse,giles2020approximating}, we call these substitutes \emph{approximate random variables}, and the originals as \emph{exact random variables}. The most well known is to use Rademacher random variables, producing what's known as the weak Euler-Maruyama scheme \citep[page~XXXII]{kloeden1999numerical}, where the Rademacher random variables have the desired mean. More advanced methods include more generalised moment matching procedures, as discussed by \citet{muller1958inverse}, and piecewise polynomial approximations and generalised approximate random variables by \citeauthor{giles2020approximate} \citep{giles2020approximate,sheridan2020approximate_inverse}. The Euler-Maruyama schemes using the exact Gaussian random variables $ Z_n $ give rise to the approximation $ \widehat{X} $, and the approximate random variables $ \widetilde{Z}_n $ produce $ \widetilde{X} $, where the schemes are respectively 
\begin{equation*}
\widehat{X}_{n+1} = \widehat{X}_n + a(t_n, \widehat{X}_n) \delta + b(t_n, \widehat{X}_n)\sqrt{\delta} Z_n
\qquad \text{and} \qquad 
\widetilde{X}_{n+1} = \widetilde{X}_n + a(t_n, \widetilde{X}_n) \delta + b(t_n, \widetilde{X}_n)\sqrt{\delta} \widetilde{Z}_n.
\end{equation*}

One of the approximations we will utilise later is the piecewise linear approximation by \citet{giles2020approximating}. This generates approximate Gaussian random variables by the inverse transform method \citep{glasserman2013monte} using a piecewise linear approximation $ \widetilde{\Phi}^{-1} $ to the Gaussian distribution's inverse cumulative distribution function $ \Phi^{-1} $. The exact construction is detailed by \citet{giles2020approximating}, although an example approximation is demonstrated in figure~\ref{fig:piecewise_linear_approximation}. A piecewise linear approximation $ \widetilde{\Phi}^{-1} \approx \Phi^{-1} $ using 8 intervals is shown in figure~\ref{fig:piecewise_linear_gaussian_approximation}, and the resultant probability density function of the approximation $ \rho $ is shown in figure~\ref{fig:piecewise_linear_gaussian_approximation_pdf}. The probability density function has compact support, and there are a few tiny inaccessible regions with zero measure, as indicated. The rounding error model we will later propose will hold for both exact Gaussian random variables, and also certain classes of approximate Gaussian random variables, which we require that they satisfy assumption~\ref{asmp:approximate_random_variables}.

\begin{figure}[htb]
\centering
\hfil
\subfigure[A piecewise linear approximation using 8 intervals.\label{fig:piecewise_linear_gaussian_approximation}]{\includegraphics{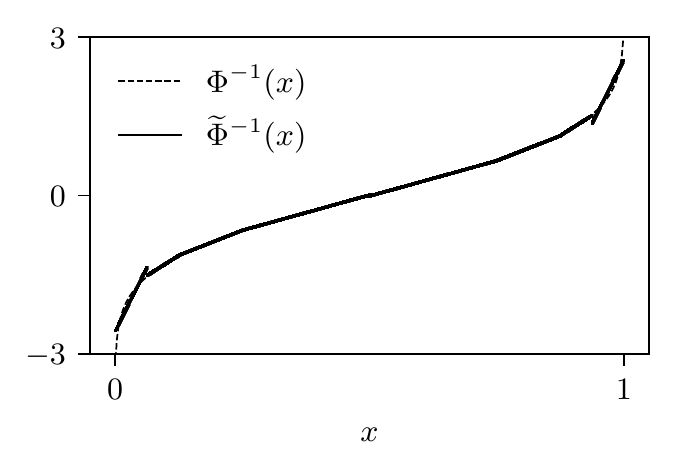}}\hfil
\subfigure[The resultant probability density function with regions of zero measure indicated.\label{fig:piecewise_linear_gaussian_approximation_pdf}]{\includegraphics{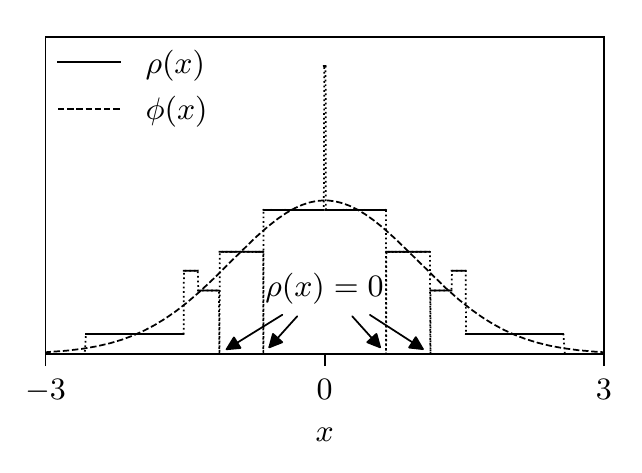}}\hfil

\caption{A piecewise linear approximation of the Gaussian distribution's inverse cumulative distribution function, and the resultant probability density function.}
\label{fig:piecewise_linear_approximation}
\end{figure}

\begin{assumption}
\label{asmp:approximate_random_variables}
Let any approximate Gaussian random variables $ \widetilde{Z} $ be zero mean, uniformly bounded, and have finite variance $ \mathbb{V}(\widetilde{Z}) = O(1) < \infty  $, and also have all higher order moments be finite. Furthermore, let there exist a corresponding probability density function $ \rho $ such that $ \mathbb{P}(\widetilde{Z} \in [z, z + \dd{z}]) = \rho(z) \dd{z}$.  Let $ \rho $ be bounded by $ K $ such that $ \rho \leq K < \infty $ and be smooth almost everywhere such that $ \rho \in C^\infty(\mathbb{R}\backslash\mathcal{M})$, where $ \mathcal{M} $ is a finite set of $ M $ points where $ \rho $ is discontinuous. Lastly, let $ \rho $ decay sufficiently fast such that for any finite constant $ \alpha $ that $ \sum_{k=-\infty}^\infty 2^{2k} \max_{y'\in[2^{k}, 2^{k+1}]} \rho(y'-\alpha) < \infty $.
\end{assumption}

\begin{lemma}
\label{lemma:exact_gaussian_distribution}
The exact Gaussian distribution satisfies assumption~\ref{asmp:approximate_random_variables}.
\end{lemma}

\begin{proof}
We immediately have that the Gaussian distribution is zero mean and has unit variance, and is uniformly bounded \citep[appendix~C.2]{blundell2014concepts}. The probability density function for the Gaussian distribution is $ \rho \equiv \phi $ where $ \phi(z) \coloneqq \tfrac{1}{\sqrt{2\pi}} {\exp}(-\tfrac{1}{2}z^2) $, which is $ C^\infty(\mathbb{R}) $ and maximal at zero where $ \rho \leq \phi(0) = \tfrac{1}{\sqrt{2\pi}} $. To show $ \sum_{k=-\infty}^\infty 2^{2k} \max_{y'\in[2^{k}, 2^{k+1}]} \rho(y'-\alpha) < \infty $ we note that the summand is increasing as $ k $ increases, and $ \rho $ will be maximal for one index $ k^* $ where $ y' = \alpha $, and for indices $ k > k^* $ the $ \rho $ term will thereafter be decreasing. Thus we can approximate the possibly divergent part of the summation by the integral 
\begin{equation*}
\sum_{k=k^* + 1}^\infty 2^{2k} \max_{y'\in[2^{k}, 2^{k+1}]} \rho(y'-\alpha) = \sum_{k=k^* + 1}^\infty 2^{2k}  \rho(2^k-\alpha)
\approx \int_{2^{k^* + 1}}^\infty  x^2  \phi(x-\alpha) \dd{x}
\leq \int_{-\infty}^\infty  x^2  \phi(x-\alpha) \dd{x} < \infty.
\end{equation*}
The summation can be bounded from above and below by similar integrals, and thus is not divergent. 
\qedhere
\end{proof}

\begin{lemma}
\label{lemma:approximate_gaussian_distribution_linear}
The approximate Gaussian distribution resulting from the piecewise linear approximation by \citet{giles2020approximating} satisfies assumption~\ref{asmp:approximate_random_variables}.
\end{lemma}

\begin{proof}
For a finite number of approximation intervals, we can see from figure~\ref{fig:piecewise_linear_gaussian_approximation_pdf} that the probability density function is symmetric, uniformly bounded, has compact support and thus finite variance. The number of discontinuities is finite, and as $ \rho $ has compact support it immediately satisfies the summation bound from assumption~\ref{asmp:approximate_random_variables}. \qedhere
\end{proof}

\begin{lemma}
\label{lemma:approximate_gaussian_distribution_cubic}
The approximate Gaussian distribution resulting from the piecewise cubic approximation by \citet{giles2020approximating} satisfies assumption~\ref{asmp:approximate_random_variables}.
\end{lemma}

\begin{proof}
The proof follows identically to the proof of lemma~\ref{lemma:approximate_gaussian_distribution_linear}. \qedhere
\end{proof}

The motivation for introducing these approximate random variables was increased simulation speed. However, for the ultimate speed improvements, it is desirable to both switch to approximate random variables, and simultaneously decrease the arithmetic precision used, giving a twofold speed improvement. Reducing the precision alone has been explored with applications to field programmable gate arrays  \citep{brugger2014mixed,omland2015exploiting,omland2016mixed,chow2012mixed}, as have multilevel Monte Carlo schemes using varying fidelities of approximate random variables \citep{muller1958inverse}. However, performing both simultaneously is touched upon by \citet{giles2019random_multilevel}, although using a quite restrictive truncated uniform random bit Monte Carlo algorithm, and extensions to more general approximation schemes without varying the precision is done by \citet{giles2020approximate}. However, the work by \citet{giles2019random_multilevel} is primarily a cost most, and does not model the effect of rounding error. Thus, while our contribution is an extension of these works, it is an important and novel demonstration and vindication of the utility of low precisions with approximate random variables.

In order to describe the effect of rounding error resulting from the Euler-Maruyama scheme, the two most prominent works are by \citet{arciniega2003rounding} and \citet{omland2016mixed}, which are models for the average and worst case errors respectively. \citet{arciniega2003rounding} provide an \textit{ad hoc} statistical model and analysis for the rounding error arising from finite precision floating point calculations within the Euler-Maruyama scheme. Denoting the estimate produced when working in finite precision as $ \overline{X} $, they propose that at the $ n $-th iteration, all of the composite floating point arithmetic in the Euler-Maruyama update culminates in an additive error $ \varepsilon_n $ where
$ \overline{X}_{n+1} = \overline{X}_n + a(t_n, \overline{X}_n) \delta + b(t_n, \overline{X}_n)\sqrt{\delta} Z_n + \varepsilon_n $.
This error is assumed to follow a Gaussian distribution, be zero mean, and have a variance $ \mathbb{V}(\varepsilon) \leq C \varrho^2 $, where $ \varrho $ is the unit roundoff and $ C $ is some arbitrary constant. The main result from their analysis \citep[theorem~2.2]{arciniega2003rounding} is $ \mathbb{E}(\lvert \widehat{X}_N - \overline{X}_N \rvert^2) \leq  CN\varrho^2 $. The model from \citet{omland2016mixed} uses a more rigorous finite precision framework. For brevity, letting $ \oplus $ and $ \otimes $ represent floating point addition and multiplication, then the model by \citet{omland2016mixed} in effect considers $ \overline{X}_{n+1} = \overline{X}_n \oplus ((a(t_n, \overline{X}_n) \otimes \delta) \oplus ((b(t_n, \overline{X}_n)\otimes \sqrt{\delta})\otimes Z_n)) $ and produces the worst case bound $ \mathbb{E}(\lvert \widehat{X}_N - \overline{X}_N \rvert^2) \leq  CN^2\varrho^2 $ \citep[theorem~4.8]{omland2016mixed}. The model we will present will take the model from \citet{omland2016mixed} as its starting point, but under assumptions appropriate for the Euler-Maruyama scheme, will ultimately reduce to a model closer resembling that by \citet{arciniega2003rounding}.

\subsection{A leading order error model}
\label{sec:a_leading_order_error_model}

Starting with the more fundamentally rooted model from \citet{omland2016mixed}, we expect to recover, under appropriate assumptions, the more statistically motivated model from \citet{arciniega2003rounding}. To this end we look to expand the model from \citet{omland2016mixed} and see the effects of arithmetic roundoff within the Euler-Maruyama update. Before presenting the analysis, we will briefly introduce a small amount of floating point notation. 

Numbers used in calculations must be stored in finite precision, where we denote the set of representable numbers as $ \overline{\mathbb{R}} \subset \mathbb{R}$, where we introduce the rounding operator $ R \colon \mathbb{R} \to \overline{\mathbb{R}}$ which implements the desired rounding mode, which we assume is round to nearest even. For finite precision binary arithmetic operations $ \circledast $ (such as $\oplus$, $\otimes$, etc.), then for two floating point numbers $ x,y \in \overline{\mathbb{R}} $, we assume the standard rounding model by \citet[2.2, (2.4)]{higham2002accuracy} \citep[page~99, (13.7)]{trefethen1997numerical} that $ x \circledast y = R(x\ast y) = (x \ast y ) (1 + \varepsilon)$ where $ \lvert \varepsilon\rvert \leq \varrho $. 

We begin by expressing the finite precision Euler-Maruyama update as  
$ \overline{X}_{n+1} = \overline{X}_n \oplus (A_n \oplus B_n ) $, where $ A_n \coloneqq 
a(t_n, \overline{X}_n) \otimes \delta $ and $ B_n \coloneqq  (b(t_n, \overline{X}_n) \otimes \sqrt{\delta}) \otimes \widetilde{Z}_n $. For an appropriately non dimensionalised stochastic process, such that $ T = 1 $, $ X_0 = O(1) $, $ a = O(1) $, $ b = O(1) $, then we anticipate $ \overline{X}_n = O(1) $, $ \widetilde{Z}_n = O(1) $, $ A_n = O(\delta) $, and $ B_n = O(\sqrt{\delta}) $. Similarly, for the precision levels and discretisations we have $ \varrho \ll 1 $ and $ \delta \ll \sqrt{\delta} \ll 1 $. This then gives us the size ordering $ \mathbb{E}(\lvert A_n \rvert) \ll \mathbb{E}(\lvert B_n \rvert) \ll \mathbb{E}(\lvert \overline{X}_n \rvert)$.

The first addition will produce an absolute error $ \eta'_n $ from $ A_n \oplus B_n = A_n + B_n + \eta'_n  $, where $ \eta'_n $ will be of  a size comparable with the unit roundoff and the larger of $ A_n $ and $ B_n $, which is $ B_n $. As $ b $ is assumed to have linear growth we obtain $ \lvert \eta'_n\rvert  \sim \lvert B_n \varrho\rvert  = O(\sqrt{\delta} \varrho (1 + \lvert\overline{X}_n\rvert))$, and after performing this first floating point addition we will be left with
$ \overline{X}_{n+1} = \overline{X}_n \oplus (B_n + A_n + \eta'_n ) $,
where we have written $ B_n + A_n + \eta'_n $ in order of decreasing magnitudes. 

For the remaining addition operation, as we $ X_n = O(1) $ and $ \mathbb{E}(\lvert \overline{X}_n \rvert)  \gg \mathbb{E}(\lvert B_n \rvert)  $, the nett result from the remaining floating point addition then is that this will produce a second absolute arithmetic error $ \eta_n $ where $ \lvert \eta_n \rvert  \sim \lvert  \overline{X}_n \varrho \rvert = O(\varrho (1 + \lvert \overline{X}_n\rvert ))$, and the Euler-Maruyama update will become
$ \overline{X}_{n+1} \approx \overline{X}_n + B_n + A_n + \eta_n + \eta'_n $, 
where again we have written the contributions in order of decreasing magnitudes. We identify two dominant sources of error. The first is $ \eta'_n $, arising from the addition of the drift term to the diffusion term. The second is $ \eta_n $, arising from the addition of this sum to the underlying process.  We expect $ \lvert \eta'_n\rvert  = O(\varrho\sqrt{\delta} (1 + \lvert \overline{X}_n\rvert ))$ and $ \lvert \eta_n\rvert = O(\varrho (1 + \lvert \overline{X}_n\rvert ))$, and thus $ \mathbb{E}(\lvert \eta_n\rvert)  \gg \mathbb{E}(\lvert \eta'_n\rvert ) $.  If we then allow for the inclusion of other higher order contributions, we can then see that we expect to obtain
$ \overline{X}_{n+1} = \overline{X}_n + B_n + A_n + \eta_n + \eta'_n + \eta''_n $.

The \citet{arciniega2003rounding} model assumes that only $ \eta_n $ is significant, and makes the assertion that this can be modelled as a zero mean Gaussian random variable with a variance only proportional $ \varrho^2 $. In our model, we will more rigorously justify the zero mean nature, drop the requirement that this exactly follows a Gaussian distribution, and show that the smaller second order contributions from $ \eta' $ are not necessarily negligible, but contribute to the nett rounding error at the same rate as the leading order $ \eta $ process. We propose model~\ref{model:rounding_errors} as an appropriate model for the rounding errors arising in the Euler-Maruyama scheme.  

\begin{model}
\label{model:rounding_errors}
Let the Euler-Maruyama scheme use random variables $ \widetilde{Z} $ which satisfy assumption~\ref{asmp:approximate_random_variables}. The composite effects of rounding error introduce two dominant sources of error, $ \eta $ and $ \eta' $, where at each step we have 
\begin{equation*}
\overline{X}_{n+1} = \overline{X}_n + a(t_n, \overline{X}_n) \delta +  b(t_n, \overline{X}_n) \sqrt{\delta} \widetilde{Z}_n + \eta_n + \eta'_n.
\end{equation*}
The larger of these is $ \eta_n = O(\varrho (1 + \lvert\overline{X}_n\rvert)) $, which is a martingale increment, and the smaller of these  is $ \eta'_n = O(\varrho\sqrt{\delta} (1 + \lvert \overline{X}_n\rvert)) $, which is a possibly non martingale increment.
\end{model}

Inspecting model~\ref{model:rounding_errors}, the key modelling assumption requiring justification is the martingale nature of $ \eta_n $. We already justified in our discussions that $ \eta'_n = O(\varrho\sqrt{\delta} (1 + \lvert \overline{X}_n\rvert)) $, and so claiming this is a non martingale increment is no further restriction. It is straightforward to reason that $ \mathbb{E}(\lvert \eta_n\rvert) = O(\varrho (1 + \mathbb{E}(\lvert\overline{X}_n\rvert))) $, which if we take $ \mathbb{E}(\lvert \overline{X}_n\rvert ) = O(1) $ simplifies to $ \mathbb{E}(\lvert \eta_n\rvert ) = O(\varrho) $. Thus, to justify $ \eta_n $ being a martingale increment it is sufficient to reason that $ \lvert \mathbb{E}(\eta_n)\rvert \ll \mathbb{E}(\lvert \eta_n\rvert )$, which we achieve through lemma~\ref{lemma:leading_order_error}.

\begin{lemma}
\label{lemma:leading_order_error}
Assuming $ \mathbb{E}(\lvert \overline{X}_n\rvert ) = O(1) $ and the random variables $ \widetilde{Z} $ satisfy assumption~\ref{asmp:approximate_random_variables}, then under the round to nearest even rounding mode, the leading order rounding error $ \eta_n $ has $ \mathbb{E}(\eta_n) = O(\varrho^2) $.
\end{lemma}

\begin{proof}
The operation producing $ \eta_n $ is the floating point addition between $ \overline{X}_n $ and $ (A_n \oplus B_n) $.  Dropping the subscripts for brevity, we denote this by $ \alpha \oplus R(\beta) $, where $ \alpha \coloneqq  \overline{X}_n = O(1) $ and   $ \beta \coloneqq A_n + B_n = O(\sqrt{\delta}) $. The absolute rounding error $ \eta $ is then given by 
\begin{equation*}
\eta \coloneqq (\alpha \oplus R(\beta)) - (\alpha + R(\beta)) \equiv (R(\alpha + \beta) - (\alpha + \beta)) - (R(\beta) - \beta) +  (R(\alpha + R(\beta)) - R(\alpha + \beta)), 
\end{equation*}
where we will bound the three parenthesised differences in turn. 

Inspecting the first term, we can see this is the absolute error resulting from rounding the quantity $ \alpha + \beta $. Without much loss of generality, as we have assumed $ \alpha = O(1) $, let us suppose $ \alpha \in (1, 2) $. Using IEEE floating point representation, the set of representable numbers $ (1, 2) \cap \overline{\mathbb{R}} $ will all be equally spaced. Defining the quantity $ z \coloneqq \alpha + \beta $, this will fall inside some interval $I_y \coloneqq [y-\varsigma, y + \varsigma] $, where $ y\pm\varsigma $ are adjacent floating point numbers. Without loss of generality, we assume $ y - \varsigma $ is odd and $ y + \varsigma $ is even, where we either have $ z < y $ and we round down, or $ z \geq y $ and we round up, as depicted in figure~\ref{fig:round_to_nearest_even_error}.

\begin{figure}[htb]
\centering
\hfil
\subfigure[Possible values for $ z $ in $ I_y $, denoting $ z $ with ``$ \square $''.\label{fig:round_to_nearest_even_error}]{\includegraphics{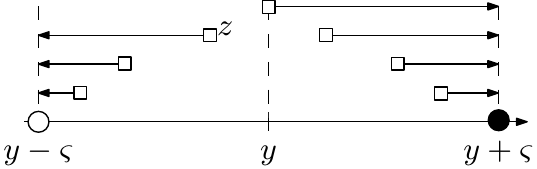}}\hfil
\subfigure[Possible values for $ \zeta $ in $ I'_y $, demonstrating when the round to nearest even tie break causes rounding error by shaded regions. We denote $ \zeta $ with ``$ \blacksquare $'', and $ z $ with ``$ \square $''.\label{fig:round_to_nearest_even_tie_break_error}]{\includegraphics{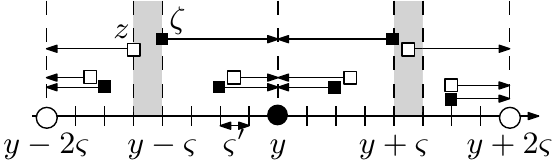}}\hfil

\caption{Rounding to the nearest even. We denote even representable values using ``\raisebox{-0.25em}{\Huge$ \bullet $}'', odd values using ``\raisebox{0.05em}{$ \bigcirc $}''. Arrows show the values rounded to.}
\label{fig:round_to_nearest_even}
\end{figure}

We can then evaluate the expectation $ \mathbb{E}((R(z) - z)\indicatorfn_{\{z \in I_y\}}) $ where we have 
\begin{equation*}
\mathbb{E}((R(z) - z)\indicatorfn_{\{z \in I_y\}}) 
= \mathbb{E}((R(z) - z)\indicatorfn_{\{z \in [y-\varsigma, y)\}}) 
+ \mathbb{E}((R(z) - z)\indicatorfn_{\{z \in [y, y+\varsigma]\}}).
\end{equation*}
These expectations can be written as integrals, giving
\begin{equation*}
\mathbb{E}((R(z) - z)\indicatorfn_{\{z \in I_y\}}) 
= \int_{y-\varsigma}^{y} ((y-\varsigma) - z) \mathbb{P}(\dd{z}) 
+ \int_{y}^{y+\varsigma} ((y+\varsigma) - z) \mathbb{P}(\dd{z}).
\end{equation*}
At time $ t_n $, the variable $ \overline{X}_n $ is $ \mathcal{F}_n $-measurable, but $ Z_n $ is $ \mathcal{F}_{n+1} $-measurable, and thus $ z $ and $ \beta $ will have the same distribution as $ \widetilde{Z} $. Denoting the probability density function of $ \beta $ as $ \rho $, which by extension satisfies assumption~\ref{asmp:approximate_random_variables}, we obtain
\begin{equation*}
\mathbb{E}((R(z) - z)\indicatorfn_{\{z \in I_y\}}) 
= \int_0^{\varsigma} (\varsigma - z) (\rho(y- \alpha + z) - \rho(y- \alpha -z)) \dd{z}.
\end{equation*}
As $ \varsigma \ll 1 $, if $ \rho $ is smooth everywhere in the interval $ I_y $, then we can approximate this using a Taylor series expansion, otherwise we use the bound from assumption~\ref{asmp:approximate_random_variables}, obtaining
\begin{equation*}
\lvert \mathbb{E}((R(z) - z)\indicatorfn_{\{z \in I_y\}}) \rvert 
\leq  \begin{cases}
C \varsigma^3 \lvert \rho'(y - \alpha)\rvert  + O(\varsigma^5 \lvert \rho'''(y-\alpha)\rvert )& \text{if } I_y \cap \mathcal{M} = \emptyset \\
C \varsigma^2 K & \text{if } I_y \cap \mathcal{M} \neq \emptyset,
\end{cases}
\end{equation*}
for some arbitrary constant $ C $.

Using this expectation in the law of total expectation, then in the limits $ I_y \to  \dd{I_y} $ we obtain
\begin{equation*}
\mathbb{E}(R(z) - z)  
=  \mathbb{E}(\mathbb{E}(R(z) - z\mid z\in I_y))  
\approx \int_{\mathbb{P}(I_y) > 0} \dfrac{ \mathbb{E}((R(z) - z)\indicatorfn_{\{z \in I_y\}})}{\mathbb{P}(I_y)} \mathbb{P}(\dd{I_y}) 
\approx \int  \dfrac{ \mathbb{E}((R(z) - z)\indicatorfn_{\{z \in I_y\}})}{2\varsigma}  \dd{y},
\end{equation*}
where in the last approximation we used $ \mathbb{P}(I_y) \approx \rho(y - \alpha) 2\varsigma $ and $ \mathbb{P}(\dd{I_y}) \approx \rho(y - \alpha) \dd{y} $. In the limit $ \delta \ll 1 $ then we can approximate our integration domain as $ y \in  [0, \infty) $, giving
\begin{equation*}
\lvert \mathbb{E}(R(z) - z) \rvert 
\leq 
\left\lvert \int_0^\infty C \varsigma^2 \rho'(y - \alpha) \indicatorfn_{\{I_y\cap \mathcal{M} = \emptyset\}} \dd{y} \right\rvert 
+ \int_0^\infty C \varsigma K \indicatorfn_{\{I_y\cap \mathcal{M} \neq \emptyset\}}\dd{y}.
\end{equation*}
In IEEE representation, $ \varsigma $ is a constant between powers of two, and thus $ 2\varsigma = \varrho 2^k $  for $ k \in \mathbb{Z} $. Consequently, the first integral is readily decomposed into sub intervals where $ \varsigma $ is a constant. For the second integral we let $ I^*_m $ denote the interval containing the discontinuity at position $ m \in \mathcal{M} $. These discontinuities occur at the corresponding $ k $ values $ k_m $ such that $ I^*_m \subset [2^{k_m}, 2^{k_m + 1}] $, and so we obtain
\begin{equation*}
\lvert \mathbb{E}(R(z) - z) \rvert  
\leq C \left\lvert \sum_{k=-\infty}^{\infty} \varrho^2  2^{2k}\int_{2^k}^{2^{k+1}} \rho'(y - \alpha) \indicatorfn_{\{I_y\cap \mathcal{M} = \emptyset\}} \dd{y} \right\rvert 
+ C K \sum_{m \in \mathcal{M}} \varsigma \int_{I^*_m}  \dd{y}. 
\end{equation*}
The first integral can largely be evaluated exactly. The integration domain will contain at most $ M $ intervals containing singularities, and thus $ M+1 $ subintervals of the form $ \int_{a}^{b} \rho'(y-\alpha) \dd{y} $ where $ \rho' $ is continuous in the domain $ [a,b] $, and thus $ \int_{a}^{b} \rho'(y-\alpha) \dd{y} = [\rho(y-\alpha)\rvert_{a}^{b} $. We then use the bound $  \lvert \int_{a}^{b} \rho'(y-\alpha) \dd{y}\rvert \leq \lvert \rho(a-\alpha) \rvert  + \lvert \rho(b-\alpha) \rvert \leq 2 \max_{y'\in[a,b]} \rho(y' - \alpha) $. For the second integral we use $ \int_{I^*_m}  \dd{y} = \varsigma $ to give
\begin{equation*}
\lvert \mathbb{E}(R(z) - z) \rvert  \leq C \varrho^2 \sum_{k=-\infty}^{\infty}   2^{2k} (M+1) \max_{y'\in[2^{k}, 2^{k+1}]} \rho(y' - \alpha) 
+ C K \varrho^2 \sum_{m \in \mathcal{M}} 2^{2k_m} = O(\varrho^2), 
\end{equation*}
where in the last equality we used assumption~\ref{asmp:approximate_random_variables}. This shows the desired $ O(\varrho^2) $ holds for the first parenthesised error constituting $ \eta $. The same line of reasoning similarly holds the for $ R(\beta) - \beta $ term also constituting $ \eta $ (akin to taking $ \alpha \to 0 $), arriving at an identical limiting bound.

It remains to bound the final $ R(\alpha + R(\beta)) - R(\alpha + \beta) $ term constituting $ \eta $. Unlike the previous two terms, we will see that this term only contributes a rounding error when the round to nearest even tie break rule is required, and in most scenarios $ \alpha + R(\beta)$ and $\alpha + \beta $ will round to the the same number. To tackle this final term we introduce the slightly larger interval $ I'_y \coloneqq [y-2\varsigma, y+2\varsigma] $, where $ y-2\varsigma $, $ y $, and $ y + 2\varsigma $ are all representable and adjacent. Without loss of generality we assume $ y $ is even and $ y \pm 2\varsigma $ are odd. Given $ \beta = O(\sqrt{\delta}) $ and $ \alpha  = O(1) $, we know that $ \lvert \beta - R(\beta)\rvert \leq \varsigma' $ where $ \varsigma' \ll \varsigma $, and thus $ \beta  $ is rounded first on a much finer granularity than $ \alpha + \beta $. Keeping our definition $ z \coloneqq \alpha + \beta $ and introducing $ \zeta \coloneqq \alpha + R(\beta) $, then the discrete set of values $ \zeta $ can take has a much finer granularity than the three representable numbers in $ I'_y $, namely $ \zeta \in \{y \pm n \varsigma'\} \cap I'_y$ for integers $ n \in \mathbb{N} $. We display the set of values $ \zeta $ can take in $ I'_y $ in figure~\ref{fig:round_to_nearest_even_tie_break_error}, where we demonstrate several possible rounding scenarios. 

Inspecting figure~\ref{fig:round_to_nearest_even_tie_break_error}, we can see that in most situations $ \zeta  $ and $ z $ round to the same number. These only round to different numbers when $ \zeta $ lies on a tie break value and $ z $ takes a different value that is rounded to an odd number, as indicated by the shaded regions in figure~\ref{fig:round_to_nearest_even_tie_break_error}. Thus, for the final term we have the expectation
\begin{equation*}
E((R(\zeta) - R(z))\indicatorfn_{\{z \in I'_y\}}) = 
E((R(\zeta) - R(z))\indicatorfn_{ \{z \in [y - \varsigma - \varsigma', y - \varsigma)\} } \indicatorfn_{\{\zeta = y - \varsigma\}})  +  E((R(\zeta) - R(z)) \indicatorfn_{ \{z \in (y + \varsigma, y + \varsigma + \varsigma']\} } \indicatorfn_{\{\zeta = y + \varsigma\}}).
\end{equation*}
The first expectation will round $ \zeta \to y $ and $ z \to y - 2\varsigma $, giving a nett rounding error of $ 2\varsigma $, and the second expectation will round $ \zeta \to y $ and $ z \to y + 2\varsigma $ giving a rounding error of $ -2\varsigma $. Thus, expressing these expectations as integrals we obtain
\begin{equation*}
E((R(\zeta) - R(z))\indicatorfn_{\{z \in I'_y\}})  = \int_{y - \varsigma - \varsigma'}^{y - \varsigma} 2\varsigma \mathbb{P}(\dd{z}) + \int_{y + \varsigma}^{y + \varsigma + \varsigma'} -2\varsigma \mathbb{P}(\dd{z}) = 2\varsigma \int_0^{\varsigma'} (\rho(z + y - \varsigma - \varsigma' - \overline{a}) - \rho(z + y + \varsigma - \overline{a})) \dd{z}.
\end{equation*}
For this final integral expression we can again either perform a Taylor series expansion or use our bounds from assumption~\ref{asmp:approximate_random_variables} to obtain to leading order
\begin{equation*}
\lvert E((R(\zeta) - R(z))\indicatorfn_{\{z \in I'_y\}}) \rvert \leq 
\begin{cases}
C \varsigma (\varsigma \varsigma' + (\varsigma')^2 ) \lvert \rho'(y-\alpha) \rvert  & \text{if } I'_y \cap \mathcal{M} = \emptyset \\
C\varsigma\varsigma' K & \text{if } I'_y \cap \mathcal{M} \neq \emptyset.
\end{cases}
\end{equation*}
As $ \varsigma' \ll \varsigma $, we see that this bound is equivalent to that found earlier for the other two terms constituting $ \eta $. Again, by using the law of total expectation and the same steps as before we obtain the same $ O(\varrho^2) $ bound. Combining the three bounds completes the proof. \qedhere
\end{proof}

While model~\ref{model:rounding_errors} is justified by lemma~\ref{lemma:leading_order_error}, we can appreciate that the proof of lemma~\ref{lemma:leading_order_error} makes use of several \textit{ad hoc} approximations and limiting cases. However, as our ultimate aim is only to justify our model, rather than derive it, such conveniences are permissible. This serves to illustrate from first principles why the leading order error term $ \eta $ is effectively zero mean. Overall then, lemma~\ref{lemma:leading_order_error} provides a much more rigorous justification of the zero mean nature of the leading order error than was simply asserted in the  \citet{arciniega2003rounding} model. 

To illustrate how the bound for the final nett rounding error is produced from our model for the incremental rounding error, we recall a convenient lemma from \citeauthor{giles2020approximate} \citep[lemma~4.3]{giles2020approximate} \citep[lemma~5.2.3]{sheridan2020nested}, which we present without proof as lemma~\ref{lemma:strong_error_bound}.

\begin{lemma}
\label{lemma:strong_error_bound}
Suppose for a process $ \mathcal{E}_n $ we have 
$ \mathcal{E}_{n+1} = \mathcal{E}_n + \delta \mathcal{A}_n + \sqrt{\delta} \widetilde{Z}_n \mathcal{B}_n + \Xi_n + \Theta_n $,
using a discretisation interval $ \delta $. 
We assert $ \mathcal{E}_0 = 0 $ almost surely, $ \widetilde{Z}_n $ are i.i.d.\ zero mean random variables with all finite moments bounded, and $ \mathcal{A}_n $ and $ \mathcal{B}_n $ are $ \mathcal{F}_{t_n} $-adapted with $ \lvert \mathcal{A}_n\rvert \leq L_A \lvert \mathcal{E}_n\rvert  $ and $ \lvert \mathcal{B}_n\rvert \leq L_B \lvert \mathcal{E}_n\rvert  $ for some strictly positive and finite constants $ L_A $ and $ L_B $. The process $ \Xi $ is a martingale where $ \mathbb{E}(\Xi_n\mid\mathcal{F}_{t_n}) = 0 $, and for integers $ p \geq 2 $  and a constant $ s \in \mathbb{R} $ there are finite and strictly positive constants $ c_1 $ and $ c_2 $ such that $ \mathbb{E}(\lvert\Xi_n\rvert^p) \leq c_1 \delta^{p(s + 1/2)} $, and similarly $ \mathbb{E}(\lvert\Theta_n\rvert^p) \leq c_2 \delta^{p(s + 1)} $ for $ \Theta $. Then there exists constants $ c_3 $ and $ c_4 $ which depends only on $ L_A $, $ L_B $, and $ p $ such that 
$ \mathbb{E}(\sup_{n \leq N} \lvert\mathcal{E}_n\rvert^p) \leq c_3(c_4 c_1 + c_2)\delta^{ps} $, 
where $ c_4 = 18p^{3/2}(p - 1)^{-3/2} $.
\end{lemma}

\begin{proof}
The proof is given by \citeauthor{giles2020approximate} \citep[lemma~4.3]{giles2020approximate} \citep[lemma~5.2.3]{sheridan2020nested}, and proceeds by a combination of Jensen's inequality, the discrete Burkholder-Davis-Gundy inequality, and the discrete Gr\"{o}nwall inequalities. \qedhere
\end{proof}

By considering the difference between the process $ \overline{X}_n $ calculated in high and low precision, then the result $ \mathbb{E}(\lvert \widehat{X}_N - \overline{X}_N \rvert^2) \leq  CN\varrho^2 $ from \citet[theorem~2.2]{arciniega2003rounding} immediately follows from lemma~\ref{lemma:strong_error_bound}. Furthermore, we obtain an identical bound from lemma~\ref{lemma:strong_error_bound} for the nett error that arises from model~\ref{model:rounding_errors}. 

\begin{lemma}
\label{lemma:rounding_error_two_way}
Using model~\ref{model:rounding_errors} for the rounding errors, then 
$ \mathbb{E}(\lvert \widehat{X}_N - \overline{X}_N \rvert^2) \leq  CN\varrho^2 $.
\end{lemma}

\begin{proof}
Defining $\mathcal{E}_n \coloneqq \widehat{X}_n - \overline{X}_n $ and differencing the appropriate Euler-Maruyama schemes for $ \widehat{X}_n $ and $ \overline{X}_n $ we obtain 
\begin{equation*}
\mathcal{E}_{n+1} = \mathcal{E}_n + \delta \underbrace{(a(t_n, \widehat{X}_n) - a(t_n, \overline{X}_n))}_{\mathcal{A}_n} {} + \sqrt{\delta} \widetilde{Z}_n \underbrace{(b(t_n, \widehat{X}_n) - b(t_n, \overline{X}_n))}_{\mathcal{B}_n} {}  + \underbrace{\sqrt{\delta} b(t_n, \widehat{X}_n) (\widehat{Z}_n - \widetilde{Z}_n) - \eta_n}_{\Xi_n} {} - \underbrace{\eta'_n}_{\Theta_n},
\end{equation*}
where we have indicated the equivalent terms in lemma~\ref{lemma:strong_error_bound}. The bounds on $ \mathcal{A}_n $ and $ \mathcal{B}_n $ follow from the standard assumptions of $ a $ and $ b $ being spatially Lipschitz continuous. Furthermore, for the $ \Xi_n $ term this is zero mean and has $ \mathbb{E}(\lvert \Xi_n\rvert^p) \leq O(\delta^{p/2}) + O(\varrho^p) \leq O(\varrho^p) $, corresponding to $ s = -\tfrac{1}{2} $ in lemma~\ref{lemma:strong_error_bound} and $ c_1 \propto \varrho^p $. Similarly, from model~\ref{model:rounding_errors} we have $ \mathbb{E}(\lvert \Theta_n\rvert^p) \leq C \varrho^p\delta^{p/2}$, also corresponding to $ s = -\tfrac{1}{2} $ and $ c_2 \propto \varrho^p $. Thus from lemma~\ref{lemma:strong_error_bound} we obtain $ \mathbb{E}(\lvert \widehat{X}_N - \overline{X}_N \rvert^p) \leq O(\varrho^p \delta^{-p/2}) $, which when we set $ p = 2 $ and note that $ \delta \propto \tfrac{1}{N} $ obtains the desired bound. \qedhere
\end{proof}

The significant insight provided by lemma~\ref{lemma:rounding_error_two_way}, which we saw in its proof when we applied lemma~\ref{lemma:strong_error_bound}, is that the nett contributions from the $ \eta $ and $ \eta' $ processes were grow at the same rate. Although the $ \eta' $ process may be smaller than $ \eta $ by a factor of $ \sqrt{\delta} $ in model~\ref{model:rounding_errors}, because it is not zero mean, its contributions do not cancel, and thus build up at a faster rate. Furthermore, this demonstrates that the updating process is permitted a systematic rounding error process, provided it is $ O(\sqrt{\delta}) $.

The variance predicted by lemma~\ref{lemma:rounding_error_two_way} is shown in figure~\ref{fig:two_way_variance_approximation}. For the approximate random variables we have used the high fidelity piecewise cubic approximation from \citet{giles2020approximating}. This is to ensure that the $ \eta_n $ term within the $ \Xi_n $ process in the proof of lemma~\ref{lemma:rounding_error_two_way} is the dominant term for moderately low precisions. Using the mpmath Python library \citep{mpmath} we adjust the number of bits used in mantissa for the low precision approximate random variables. We use 7, 10, and 23 bits, corresponding to the precisions for \texttt{bfloat16}, half, and single precisions respectively. We also consider am artificial precision using 16 bits for the mantissa to represent an intermediate precision level between half and single precision. For the stochastic process we simulate a geometric Brownian motion where $ a(t, X_t) \equiv \mu X_t $ and $ b(t, X_t) \equiv \sigma X_t $ for strictly positive constant $ \mu $ and $ \sigma $. Following the setup from \citet{giles2008multilevel} we choose $ \mu = 0.05 $, $ \sigma = 0.2 $, and $ X_0 = 1 $.

We can see from figure~\ref{fig:two_way_variance_approximation} that we approximately observe the growth in the variance for half precision anticipated by lemma~\ref{lemma:rounding_error_two_way}. However, the brain float precision seems to exhibit a variance closer to the worst case bound from \citet{omland2016mixed}. Interestingly, we see that the higher precision intermediate and single precision results appear to exhibit an approximate $ O(1) $ variance. The reason for this is because at such high precisions, the $ \eta $ term within $ \Xi_n $ is the smaller of the two terms, with the approximation error $ Z_n - \widetilde{Z}_n $ being the more dominant error. As shown by \citet{giles2020approximate} this produces a resultant error independent of the discretisation $ \delta $, and so appears $ O(1) $, as observed. Furthermore, the intermediate and single precision variances are approximately the same values, again indicating that this error is dominated by the approximate random variables' fidelity rather than their floating point precision. 

\begin{figure}[htb]
\centering
\hfil
\subfigure[Without Kahan compensated summation.\label{fig:two_way_variance_approximation}]{\includegraphics{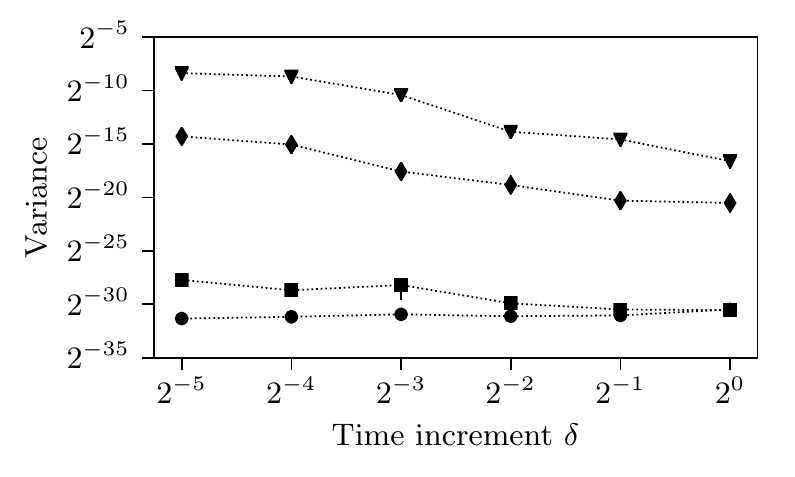}}\hfil 
\subfigure[With Kahan compensation summation.\label{fig:two_way_variance_kahan}]{\includegraphics{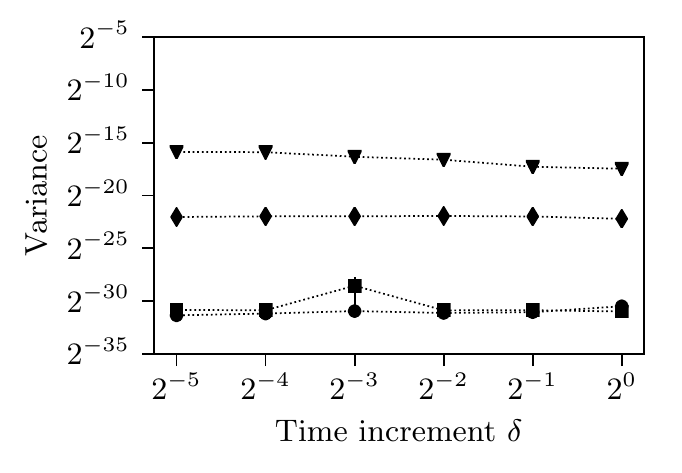}}\hfil

\caption{The variance of the difference between the exact Euler-Maruyama estimate and the estimate using low precision approximate random variables from a piecewise cubic approximation with different precisions, corresponding to lemma~\ref{lemma:rounding_error_two_way}. The precisions use differing numbers of bits for the mantissa. ($ \blacklozenge $) 7 bits. ($ \blacksquare $) 10 bits. (\raisebox{-0.1em}{\huge$ \bullet$})
16 bits. ({\large $ \blacktriangledown $}) 23 bits.}
\label{fig:two_way_variances}
\end{figure}

\subsection{Kahan compensated summation}
\label{sec:kahan_compensated_summation}

The Euler-Maruyama scheme consists of performing a cumulative summation of a sequence of update terms. The problem of summing a sequence of floating point numbers and minimising the cumulative rounding error is well known, and there have been a variety of methods developed to overcome this, such as \emph{pair wise summation} or \emph{compensated summation} \citep[4.1]{higham2002accuracy}. As the Euler-Maruyama scheme is a sequential and incremental procedure, a compensated summation is an appropriate technique. As our schemes are motivated by computational speed, then \emph{Kahan compensated summation} \citep{kahan1965further}, being the least numerically intensive, is the most suitable candidate for incorporating into the Euler-Maruyama scheme. The Kahan compensated summation adds a given increment, and then subtracts away the computed summation prior to that increment. The difference of this inferred increment from the original provides an estimate for the incurred rounding error, which is then adjusted for when adding subsequent terms in the sequence. The Kahan compensated summation procedure is outlined in algorithm~\ref{algo:kahan_compensated_summation}, and a C implementation incorporating this into the Euler-Maruyama scheme in single precision is shown in code~\ref{code:c:euler_maruyama_scheme_with_kahan_compensated_summation}. Interestingly, the related use of Kahan compensated summation in the numerical solution of ordinary differential equations was first proposed by \citet{vitasek1969numerical}, and is demonstrated by \citet[pages~86--87]{higham1993accuracy}.

\begin{algorithm}[htb]
\DontPrintSemicolon
\KwIn{A sequence $ \{x_1, x_2, \ldots, x_N \} $ of $ N $ floating point numbers.}
\KwOut{A high accuracy estimate of the summation $ \sum_{i=1}^{N} x_i $.}
Initialise both an accumulator $ a $ and compensation $ c $ to zero.\;
\For{$ x_i \in \{x_1, x_2, \ldots, x_n \} $}{
    Calculate a compensated increment $ y \leftarrow x_i - c $.\;
    Add the compensated increment $ a_{\mathrm{new}} \leftarrow a + y $ and keep the original $ a_{\mathrm{original}} \leftarrow a $.\;
    Update the compensation $ c \leftarrow (a_{\mathrm{new}} -  a_{\mathrm{original}}) - y $.\;
    Update the accumulator $ a \leftarrow a_{\mathrm{new}} $.\;
}
Use $ a $ to estimate $ \sum_{i=1}^{N} x_i $.\;
\caption[Kahan compensated summation ]{Kahan compensated summation.}
\label{algo:kahan_compensated_summation}
\end{algorithm}

\begin{lstfloat}[htb]
\begin{lstlisting}[style=C, caption={C implementation of the Euler-Maruyama scheme using Kahan compensated summation from  algorithm~\ref{algo:kahan_compensated_summation}.}, label={code:c:euler_maruyama_scheme_with_kahan_compensated_summation}]
float compensated_euler_maruyama_scheme(float X, float dX, float * compensation)
{
    float compensated_increment = dX - (*compensation); 
    float accumulated_sum = X + compensated_increment;
    (*compensation) = (accumulated_sum - X) - compensated_increment;
    return accumulated_sum;
}
\end{lstlisting}
\end{lstfloat}

An error analysis of Kahan compensated summation is provided by \citet[page~791, (3.11)]{higham1993accuracy}, \citet[Excercise~19, pages 229 and 571--573]{knuth2014art}, and \citet{goldberg1991every}. The Kahan compensated summation shown in  algorithm~\ref{algo:kahan_compensated_summation} has the overall absolute and relative error bounds of 
\begin{equation*}
(2\varrho + O(N\varrho^2))\sum_{i=1}^{N} \abs{x_i}
\qquad \text{and} \qquad
(2\varrho + O(N\varrho^2))\dfrac{\sum_{i=1}^{N} \lvert x_i\rvert}{\left\lvert \sum_{i=1}^{N} x_i\right\rvert} 
\end{equation*}
respectively. The ratio of
$ \sum_{i=1}^{N} \abs{x_i} $ to $\lvert \sum_{i=1}^{N} x_i\rvert $
is known as the condition number, representing the sensitivity of the summation to rounding error \citep{higham2002accuracy,trefethen1997numerical}. For a series of zero mean random variables, the condition number can be expected to be $ O(\sqrt{N}) $, but for several stochastic process (e.g. geometric Brownian motion) the drift term causes the increments to have a non zero mean, and hence the condition number approximately limits to a constant. 

Based on the usual error analysis of Kahan compensated summation, one might expect the leading order error from a Kahan compensated Euler-Maruyama scheme should mostly have an $ O(1) $ error dependence on $ N $, until eventually the higher order term takes effect for sufficiently large $ N $. However, inspecting model~\ref{model:rounding_errors} and code~\ref{code:c:euler_maruyama_scheme_with_kahan_compensated_summation} we see that the Kahan compensated summation is designed to tackle the leading order $ \eta $ term. However, in computing the Euler-Maruyama update, the smaller $ \eta' $ error is not compensated for, and thus will persist. Thus, even with Kahan compensated summation, we see from lemma~\ref{lemma:rounding_error_two_way} that we still expect a nett leading order rounding error $ O(\sqrt{N}\varrho) $. Overall then, we see that as we increase $ N $, we first expect for small $ N $ an $ O(\varrho) $ error, for very large $ N $ an $ O(N\varrho^2) $ error, and possibly an intermediate $ O(\sqrt{N}\varrho) $ error. 

The variances from the approximations obtained by incorporating a Kahan compensated summation into the modified Euler-Maruyama scheme are shown in figure~\ref{fig:two_way_variance_kahan}. The first thing to notice from this is that all the variances appear to be $ O(1) $, in keeping with the leading order error anticipated. For the brain float and half precision variances, there is a separation between these which is approximately a factor of $ 2^{-6} $. As we expect the leading order error to be $ O(\varrho) $, then we expect a reduction in the variance of approximately $ \tfrac{\varrho^2_{10}}{\varrho^2_{7}} \approx \approx 2^{-6} $, where we have used the subscript to denote the number of mantissa bits. Hence we can see the reduction in variance is approximately as anticipated. As for the intermediate and single precisions, we see these cluster on top of each other, again indicating that the dominant error is from approximating the Gaussian distribution rather than from finite precision arithmetic. 

\section{Multilevel Monte Carlo}
\label{sec:multilevel_monte_carlo}

In section~\ref{sec:numerical_solutions_to_stochastic_differential_equations} we introduced the usual Euler-Maruyama scheme, and a modified version which utilises approximate random variables. The reason for introducing approximate random variables was to benefit from their improved speed. This improvement will be magnified if we simultaneous also reduce the floating point precision used in the modified Euler-Maruyama scheme. After reviewing the models developed by \citet{arciniega2003rounding} and \citet{omland2016mixed} for describing the rounding error incurred during the Euler-Maruyama scheme, we have presented our own model for the rounding error, which presents a similar model to the description by \citet{arciniega2003rounding}. Furthermore, our model accounts for using approximate random variables, and is also shored up with a more rigorous mathematical justification. Ultimately though, we now have two possible types of simulation: an expensive but precise simulation using exact Gaussian random variables in a high floating point precision, and a cruder and cheaper simulation using approximate Gaussian random variables in low precision. The accuracy of the former can be combined with the speed of the latter by a multilevel Monte Carlo formulation \citep{giles2008multilevel,giles2015multilevel_review}. 

As a brief review of multilevel Monte Carlo, and how our work fits within this, let us suppose we wish to compute the expectation of some functional $ P $ which acts on the terminal solution $ X_T $ of the underlying stochastic process. The simulations can be performed using various levels of temporal discretisation, where we index the levels by $ l $. We suppose there are $ L+1 $ levels such that $ l\in \{0, 1, 2, \ldots, L\} $, where $ l = 0 $ corresponds to the coarsest possible discretisation, and $ l=L $ the finest. The approximation coming from the usual Euler-Maruyama scheme for a particular level $ l $ we denote by $ \widehat{P}_l $. Additionally, we denote those arising from the modified Euler-Maruyama scheme using high precision approximate random variables by $ \widetilde{P}_l $, and low precision approximate random variables by $ \overline{P}_l $. For notational simplicity we use the convention $ \widehat{P}_{-1} \coloneqq \widetilde{P}_{-1} \coloneqq \overline{P}_{-1} \coloneqq 0 $. \citet{giles2020approximate,giles2020approximating} suggest incorporating the approximate random variables using the nested multilevel Monte Carlo framework
\begin{equation*}
\mathbb{E}(P) 
\approx
\mathbb{E}(\widehat{P}_L) 
= 
\sum_{l = 0}^{L} \mathbb{E}(\widehat{P}_l - \widehat{P}_{l-1}) 
= 
\sum_{l = 0}^{L} \mathbb{E}(\widetilde{P}_l - \widetilde{P}_{l-1}) +  \mathbb{E}(\widehat{P}_l - \widehat{P}_{l-1} - \widetilde{P}_l + \widetilde{P}_{l-1})
= 
\sum_{l = 0}^{L} \mathbb{E}(\overline{P}_l - \overline{P}_{l-1}) +  \mathbb{E}(\widehat{P}_l - \widehat{P}_{l-1} - \overline{P}_l + \overline{P}_{l-1}),
\end{equation*}
where the first approximation is the regular Monte Carlo procedure \citep{glasserman2013monte}, the first equality is the usual multilevel Monte Carlo decomposition \citep{giles2008multilevel},  the third equality is the nested multilevel Monte Carlo framework \citep{giles2020approximate,giles2020approximating}, and the final equality is the same nested multilevel framework utilising low precision approximate random variables.  

Importantly, for a given level $ l $, the fine path's discretisation $ \delta^{\mathrm{f}} $ and the coarse path's $ \delta^{\mathrm{c}} $ are given by $ \delta^{\mathrm{f}} = 2^{-l} $ and $ \delta^{\mathrm{c}} = 2 \delta^{\mathrm{f}} $ respectively. The coarse path's Weiner increments are produced by the pairwise summation of the fine path's Wiener increments. Furthermore, the Weiner increments $ \Delta W $ are produced using Gaussian increments $ Z $ where $ \Delta W \coloneqq \sqrt{\delta} Z $, and the Gaussian increments are produced using the inverse transform method where $ Z_n = \Phi^{-1}(U_n) $. Crucially, the exact random variables $ Z_n $ and approximate random variables $ \widetilde{Z}_n $ are tightly coupled by ensuring they are both generated using the same underlying random variables, where $ Z_n = \Phi^{-1}(U_n) $ and $ \widetilde{Z}_n = \widetilde{\Phi}^{-1}(U_n) $, with $ U_n $ being the same for both.

\citet{giles2020approximating} consider the usual multilevel estimator $ \hat{\theta} $ and the nested multilevel estimator $ \overline{\theta} $ where
\begin{equation*}
\hat{\theta}  \coloneqq \sum_{l=0}^{L} \dfrac{1}{\widehat{m}_l} \sum^{\widehat{m}_l} \widehat{P}_l - \widehat{P}_{l-1}
\qquad \text{and} \qquad
\overline{\theta} \coloneqq \sum_{l=0}^L \dfrac{1}{\overline{m}_l} \sum^{\overline{m}_l} \overline{P}_l - \overline{P}_{l-1} + \dfrac{1}{\overline{M}_l} \sum^{\overline{M}_l} \widehat{P}_l - \widehat{P}_{l-1} - \overline{P}_l + \overline{P}_{l-1},
\end{equation*}
where $ \widehat{m}_l $, $ \overline{m}_l $, and $ \overline{M}_l $ are the number of paths generated, each with a computational cost of $ \hat{c}_l $, $ \overline{c}_l $, and $ \overline{C}_l $, and variance $ \hat{v}_l $, $ \overline{v}_l $, and $ \overline{V}_l $ respectively. Letting $ \widehat{T} $ denote the total computational time to achieve a mean squared error $ \varepsilon^2 $ using the estimator $ \hat{\theta} $, and similarly $ \overline{T} $ using $ \overline{\theta} $, then \citet{giles2020approximating} show
\begin{equation*}
\widehat{T} = 2\varepsilon^{-2}\left(\sum_{l=0}^L \sqrt{\hat{v}_l \hat{c}_l}\right)^2 
\qquad \text{and} \qquad 
\overline{T} = 2\varepsilon^{-2} \left(\sum_{l=0}^L \sqrt{\overline{v}_l \overline{c}_l} + \sqrt{\overline{V}_l \overline{C}_l}\right)^2,
\end{equation*}
and hence an overall saving of
\begin{equation*}
\overline{T} 
\approx 2\varepsilon^{-2} \left(\sum_{l=0}^L \sqrt{\hat{v}_l \hat{c}_l} \left( \sqrt{\dfrac{\overline{v}_l\overline{c}_l}{\hat{v}_l\hat{c}_l}} + \sqrt{\dfrac{\overline{V}_l \overline{C}_l}{\hat{v}_l \hat{c}_l}}\right)\right)^2 
\leq \widehat{T} \max_{l \leq L} \left\{ \dfrac{\overline{v}_l\overline{c}_l}{\hat{v}_l\hat{c}_l} \left(1 + \sqrt{\dfrac{\overline{V}_l \overline{C}_l}{\overline{v}_l \overline{c}_l}}\right)^2\right\}.
\end{equation*}
When the approximation's fidelity is such that $ \overline{v}_l \approx \hat{v}_l $, the term $ \tfrac{\overline{v}_l\overline{c}_l}{\hat{v}_l\hat{c}_l} \approx \tfrac{\overline{c}_l}{\hat{c}_l}$ measures the potential time savings, and the term $ (1 + (\overline{V}_l \overline{C}_l / \overline{v}_l \overline{c}_l)^{1/2})^2 $ assesses the efficiency at realising these savings. Achieving a balance between these two is required, where the approximation should be sufficiently fast so there is the potential for large savings, but of a sufficient fidelity so the variance of the more expensive four way difference is considerably lower than the variance of the cheaper two way difference.

\citet{giles2020approximate,giles2020approximating} investigate the variance of this final four way difference when using approximate random variables assuming infinite precision arithmetic for a variety of functional types \citep[lemmas~4.10 and 4.11]{giles2020approximate} \citep[corollaries~6.2.6.2 and 6.2.6.3]{sheridan2020nested}. They find that for Lipschitz continuous and differentiable functionals that 
\begin{equation*}
\lVert \widehat{P}_l - \widehat{P}_{l-1} - \widetilde{P}_l + \widetilde{P}_{l-1}\rVert_p 
\leq O(\delta^{1/2} \lVert Z - \widetilde{Z} \rVert_{p'}) 
\end{equation*}
for some  $ p' $ such that $ 2 \leq p < p' < \infty $. Restricting our attention to functionals which are also Lipschitz continuous and differentiable, and will consider just the underlying process, taking $ P(X)\equiv X $ for simplicity. 

In an analogous manner to lemma~\ref{lemma:rounding_error_two_way} we can consider the variance of this four way difference arising in the nested multilevel Monte Carlo framework when the precision is simultaneously lowered when switching to approximate random variables, giving rise to lemma~\ref{lemma:rounding_error_four_way}.

\begin{lemma}
\label{lemma:rounding_error_four_way}
For fine and coarse path simulations constructed using approximate random variables as described by \citet{giles2020approximate}, then with rounding errors described by model~\ref{model:rounding_errors} we have 
\begin{equation*}
\mathbb{E}(\lvert \widehat{X}_l - \widehat{X}_{l-1} - \overline{X}_l + \overline{X}_{l-1}\rvert^2) 
\approx O(\delta\, \mathbb{E}(\lvert Z - \widetilde{Z} \rvert^{2 + \epsilon})) + O(\delta^{-1}\varrho^2) 
\end{equation*}
as the discretisation $ \delta $ decreases for some $ \epsilon \in (0, \infty)$.
\end{lemma}

\begin{proof}
The initial $ O(\delta\, \mathbb{E}(\lvert Z - \widetilde{Z} \rvert^{2 + \epsilon})) $ term comes from the result by \citet{giles2020approximate,giles2020approximating} taking $ P(x) \equiv X $ and $ p = 2 $. The second  $ O(\delta^{-1}\varrho^2)  $ term comes from the $ \eta $ and $ \eta' $ contributions from model~\ref{model:rounding_errors}, and their final nett contributions arise using lemma~\ref{lemma:strong_error_bound} in an identical manner to the proof of lemma~\ref{lemma:rounding_error_two_way}. \qedhere
\end{proof}

We can compare the realised variances  predicted by  \citet{giles2020approximate,giles2020approximating} when using approximate random variables in infinite precision, with those predicted by lemma~\ref{lemma:rounding_error_four_way}. We use the piecewise linear approximation by \citet{giles2020approximating}, and use half precision capable hardware (rather than emulation using the mpmath Python library). We use C code on a Nvidia Jetson AGX Xavier machine, containing a Nvidia 12 core Volta GPU and an 8 core Arm~v8.2 64 bit CPU. Both the CPU and the GPU support half precision floating point arithmetic in hardware, and so we run the code on the CPU using the \texttt{\_Float16} data type for half precision,  compiled with \texttt{gcc} and notably with the flags \texttt{-O0} and \texttt{-march=armv8.2-a+fp16}. The first flag ensures no compiler optimisations are issued, guaranteeing the Kahan compensated summation is not removed by the compiler, and the second ensures half precision data types and operations are accessible and used. The results for the variances of the various multilevel Monte Carlo terms are shown in figure~\ref{fig:four_way_variance}.

\begin{figure}[htb]
\centering

\hfil
\subfigure{\centering \begin{minipage}[c]{0.49\linewidth}
\centering
\includegraphics{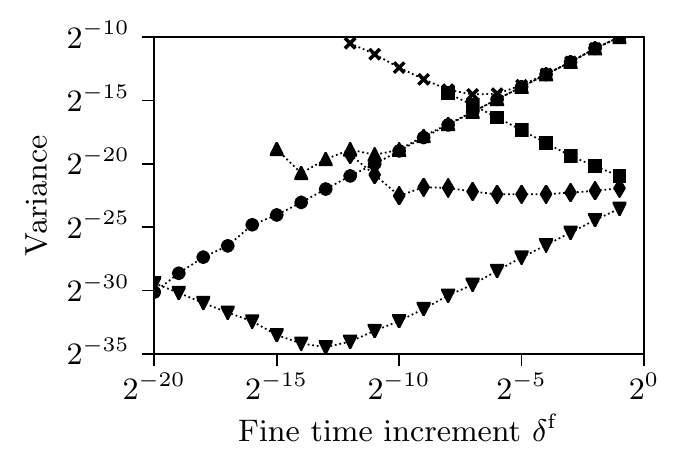}
\end{minipage}}\hfil
\subfigure{\centering
\begin{minipage}[c]{0.49\linewidth}
\centering
\renewcommand{\arraystretch}{1.4}  
\small
\begin{tabular}{cllc}
& Difference& Precision & Kahan \\ \hline 
\raisebox{-0.1em}{\huge$ \bullet$} & $ \widehat{X}_l - \widehat{X}_{l-1} $ & Double & No\\
$ \bm{\times} $ & $ \overline{X}_l - \overline{X}_{l-1} $ & Half & No\\
{\large $ \blacktriangle $} & $ \overline{X}_l - \overline{X}_{l-1} $ & Half & Yes\\
{\large $ \blacktriangledown $} & $ \widehat{X}_l - \widehat{X}_{l-1} - \overline{X}_l + \overline{X}_{l-1} $& Single & No \\
$ \blacksquare $ &$ \widehat{X}_l - \widehat{X}_{l-1} - \overline{X}_l + \overline{X}_{l-1} $ & Half & No\\
$ \blacklozenge $ & $ \widehat{X}_l - \widehat{X}_{l-1} - \overline{X}_l + \overline{X}_{l-1} $ & Half & Yes\\[3em]
\end{tabular}
\end{minipage}
}\hfil

\caption{The variance reduction when switching to approximate random variables with a nested multilevel Monte Carlo framework, showing the variances of various two and four way differences. The key indicates the difference, precision, and whether Kahan compensated summation is used.}
\label{fig:four_way_variance}
\end{figure}

Inspecting figure~\ref{fig:four_way_variance} we can make several observations. The first is that the usual two way difference (computed in double precision) exhibits the usual $ O(\delta) $ variance decay rate, as is to be expected and is a standard result \citep{kloeden1999numerical,glasserman2013monte}. 

The next item of interest is the behaviour of the four way difference computed in single precision. Down to discretisations as fine as $ \delta \approx 2^{-13} $ we can see the variance decays at the same $ O(\delta) $ rate, as already predicted and demonstrated by \citet[4.1]{giles2020approximating}. However, the novel feature predicted by lemma~\ref{lemma:rounding_error_four_way} is the emergence of the $ O(\delta^{-1}) $ rate  for very fine discretisation, arising from model~\ref{model:rounding_errors}. We can see that the onset of rounding error is not immediately catastrophic, and that for $ \delta \approx 2^{-17} $ there is still a reduction in the variance between the two and four way differences which is still approximately $ 2^{-6} $. However, eventually, for discretisations as fine as $ \delta \approx 2^{-20} $ there is no reduction in variance. Ultimately, this demonstrates that double precision is typically superfluous for Monte Carlo path simulations (ignoring sensitivity calculations for computing derivatives by finite differences), and single precision is sufficient. 

The main items of particular interest are the half precision variances. For simplicity we begin by inspecting the difference without Kahan compensated summation. The variance of the two way difference term decreases in line with the double precision two way difference down to approximately $ \delta \approx 2^{-7} $, and thereafter the effects of rounding error become dominant. This immediately places a lower limit on an uncompensated half precision framework, which is valid for discretisations coarser than $ \delta \geq 2^{-7} $. Considering the four way difference, we see that at the very coarsest level there is an appreciable variance reduction by a factor of approximately $ 2^{-12} $, although rounding error has already started to become dominant. We see that as the discretisations become finer the rounding error increases, and at $ \delta \approx 2^{-7} $ coincides with the two way difference. However, for discretisations coarser than $ \delta \geq 2^{-4} $, there is still at least a variance drop by approximately a factor of $ 2^{-6} $. This suggests that while half precision calculations may be fast, if the rounding error is not compensated for, then they are only useful on the coarsest few levels. 

If we incorporate a Kahan compensated summation to the Euler-Maruyama scheme when using half precision approximate random variables, the picture improves. The first item to note is the variance of the two way difference, which mirrors the double precision's two way variance down to approximately $ \delta \approx  2^{-11} $, placing a lower limit on the minimum possible discretisation with $ \delta \geq 2^{-11} $. For the four way difference, at the very coarsest level we see approximately the same variance as the uncompensated half precision four way difference, as might be expected. However, with the compensation, the error is an $ O(1) $ constant as the discretisation becomes ever fine down to $ \delta \approx 2^{-10} $. For discretisations $ \delta <  2^{-10} $ a higher order error process appears to dominate, and for such fine discretisations the two way and four way variances coincide. This suggests that half precision simulations using Kahan compensated summation are applicable for much finer discretisation than equivalent simulations without the Kahan compensated summation. A variance reduction of approximately $ 2^{-6} $ is achieved for discretisations $ \delta \geq 2^{-8} $. Thus Kahan compensated summation approximately doubles the scope of practical applicability for half precision simulations. 

\begin{table}[htb]
\centering
\caption{Performance of various approximations and implementations of the inverse Gaussian cumulative distribution function, and the possible speed ups offered.}
\label{tab:implementations}

\hfil 
\subfigure[The time to generate exact and approximate Gaussian random variables.\label{tab:implementation_times}]{
\begin{tabular}{llll}
Description & Precision &  Clock cycles& Source\\ 
\hline
Intel (HA) & Single &  $ 3.5 \pm 0.2 $ & \citep{giles2020approximating}\\
Piecewise linear& Single &   $ 0.5  \pm 0.1 $ & \citep{giles2020approximating} \\
Piecewise linear& Half &   $ 0.25 $ & Speculated
\end{tabular}}\hfil \hfil 
\subfigure[The maximum possible speed ups.\label{tab:implementation_speedups}]{
\begin{tabular}{ccc}
Precision & Kahan & Speed up \\
\hline
Single & No & 7 \\
Half & No & 14 \\
Half & Yes & 10
\end{tabular}
}\hfil 
\end{table}

We take the timing results on Intel AVX-512 Skylake hardware from \citet{giles2020approximating}. For the exact Gaussian distribution, we take as our baseline the single precision Intel high accuracy (HA) function. In lieu of vectorised half precision capable hardware, we speculate that for the approximate random variables, that half precision input can be processed in half the time as single precision input. Overall then we have the times shown in table~\ref{tab:implementation_times}. Furthermore, we make the idealised assumption that the cost of the simulations is entirely based on the cost of generating random numbers, and neglect the cost of all other arithmetic operations. These give the maximum potential speed ups shown in table~\ref{tab:implementation_speedups}. For the half precision approximation using Kahan compensated summation, we again speculate and suggest an intermediate value between the single and half precision offerings. 

For each possible discretisation level, the estimated speed ups predicted for each individual level from the nested multilevel Monte Carlo analysis are shown in figure~\ref{fig:four_way_savings}. We can see from this that for the bulk of levels the single precision approximation offers a good potential speed up by a factor of approximately 7. However, for the very coarsest few levels, the half precision approximations without Kahan compensated summation offer superior speed ups by a factor of 10--12. Dependent on the speed reduction that comes from incorporating the half precision Kahan compensated summation, there is the possibility of these offering a third intermediate regime where they are they optimal choice. For our speculated speed ups from table~\ref{tab:implementation_speedups} we see there appears to be such an intermediate region, although it is not overwhelmingly competitive compared to the two other alternatives. 

\begin{figure}[htb]
\centering
\includegraphics{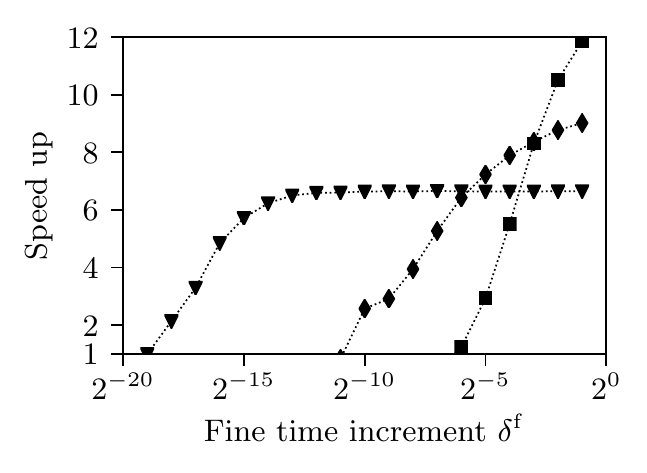}

\caption{The potential savings from a nested multilevel Monte Carlo framework using approximate random variables from a piecewise linear approximation for various discretisation levels. ({\large $ \blacktriangledown $}) Single precision. ($ \blacksquare $) Half precision without Kahan compensated summation. ($ \blacklozenge $) Half precision with Kahan compensated summation.}
\label{fig:four_way_savings}
\end{figure}

It is worth remarking that when using the Euler-Maruyama scheme, as the scheme has a strong convergence order of $ \tfrac{1}{2} $, the computational work load is approximately spread evenly over the various levels \citep{giles2008multilevel}. However, for numerical schemes with higher strong convergence orders, such as the Milstein scheme which has order 1 strong convergence \citep{kloeden1999numerical,glasserman2013monte}, the work load is predominantly concentrated on the coarsest levels \citep{giles2008multilevel}. The implication of this is that the potential half precision speed ups offered on the coarsest levels, even without Kahan compensated summation, may well dominate the multilevel savings. Hence, while half precision appears attractive even with multilevel Monte Carlo frameworks using the Euler-Maruyama scheme, this becomes even more so for the Milstein scheme and other higher order methods. 

Lastly, we can speculate about the utility of brain floats compared to regular half precision floats (taken to be the IEEE specification). Brain floats and half precision floats are both 16 bits in size, but differ in their trade off between precision and range. Brain floats have a much larger range and lower precision than regular half precision. Having a lower precision will likely mean that the initial impact of rounding error will be even more severe than for half precision. This means the variance reduction will be less favourable, and the nested multilevel Monte Carlo framework less efficient on each level, and useful over even fewer discretisations. Without cause to suspect brain floats will be any faster than regular half precision, this would suggest that while half precision is attractive for multilevel Monte Carlo applications, and brain floats may similarly be attractive for the same reasons, there is no reason to speculate that brain floats will be competitive over regular half precision floats. 

\section{Conclusions}
\label{sec:conclusions}

Performing calculations in high precisions may assuage worries about rounding errors, but makes several computations needlessly expensive. Considering the numerical simulation of stochastic differential equations, based on previous work using computationally cheap approximate random variables, we couple their incorporation into the Euler-Maruyama scheme with low precision implementations. We introduce a new model for the nett rounding error incurred which allows for systematic and unsystematic errors, analysing how the two can be balanced under appropriate assumptions. Kahan compensated summation is also discussed as a means of removing the leading order rounding error. This rounding error model is incorporated into a nested multilevel Monte Carlo scheme allowing for the speed of low precisions to be capitalised on without losing accuracy, finding that single precision is applicable for most discretisations and offers good potential speed ups, while half precision appears to offer superior speeds at only the coarsest few discretisation levels. 

Introducing finite precision calculations, we discuss the appeal of working in ever lower precisions, and the attraction of using half precision in various applications. Low precision offers improved speed by increasing bandwidth and decreasing floating point calculation times, and is rapidly gaining traction in software and hardware, primarily due to applications in machine learning. However, low precision comes with appreciable finite precision rounding error, whose effects are felt in applications including: linear algebra, machine learning, stochastic simulation, and various others. To mitigate against this there are high precision libraries, compensated summation schemes, and mathematical techniques such as Richardson extrapolation, although the problem of rounding error can be severe for 16 bit floating point formats such as half precision or brain floats.

Considering the setting of stochastic simulations and approximating solutions of stochastic differential equations using the Euler-Maruyama scheme we introduce model~\ref{model:rounding_errors} as a novel description for the effects of rounding error. Similar previous models for rounding error in this setting are notably from \citet{arciniega2003rounding} and \citet{omland2016mixed}, for the average and worst case scenarios respectively. The model we introduce is based on the framework by \citet{omland2016mixed} and recovers a similar model to the one by \citet{arciniega2003rounding}. However, comparing our model to that from \citet{arciniega2003rounding}, ours has several benefits. Firstly, ours facilitates the incorporation of approximate random variables, which \citet{giles2020approximating} showed offers considerable potential speed improvements. Our model rigorously justifies a leading order zero mean contribution, and permits a smaller second order and possibly non zero mean contribution, thus facilitating systematic and non systematic errors. Using lemma~\ref{lemma:strong_error_bound} we show the nett contributions from these two terms grow at the same rate for the Euler-Maruyama scheme, providing novel insight into the size permissible for systematic errors. 

Reviewing Kahan compensated summation, we discuss the leading order constant error this can be expected to produce. We numerically simulated geometric Brownian motion processes, finding the errors predicted by lemma~\ref{lemma:rounding_error_two_way} are observed empirically, alongside the anticipated error resulting from incorporating Kahan compensated summation. Furthermore, while Kahan compensated summation cancels the leading order rounding error, there is a balance between this and the error resulting from using approximate random variables, which becomes the dominant source of error in higher precision simulations. 

Low precision implementations of approximate random variables are best utilised with a nested multilevel Monte Carlo framework, paralleling the setup by \citet{giles2020approximating}. This has the capabilities of realising the potential speed ups offered from low precision calculations without losing any accuracy. The nested multilevel Monte Carlo setup produces a four way difference, whose variance we predict in lemma~\ref{lemma:rounding_error_four_way} and empirically observe. As expected, we find the errors resulting from introducing approximate random variables and also from low precision calculations are orthogonal effects. Using the nested multilevel Monte Carlo framework, we calculate that for a very wide range of possible discretisation levels that single precision simulations offer a speed up by a factor of approximately 7, as already shown by \citet{giles2020approximating}. However, for the very coarsest few levels, we demonstrated that half precision calculations, despite incurring significant rounding error from the very onset, can be successfully utilised within a multilevel Monte Carlo scheme. On these coarsest few levels, half precision may offer speed improvements by a factor of 10--12. For the Euler-Maruyama scheme the work load is usually evenly spread across the various levels, whereas for the higher order Milstein scheme the work load is dominated by the coarsest few levels. Thus we have been able to demonstrate the utility and applicability of half precision approximate random variables for stochastic simulation applications. Lastly, dependent on the cost model for the added arithmetic required for Kahan compensated summation, we were also able to demonstrate that this may provide a further intermediate region, extending the range of half precision to even finer discretisation levels, compounding their benefits. 

All of the code to produce these figures is freely available and hosted by \citet{sheridan2020low_precision}. 

\section{Acknowledgements}
\label{sec:acknowledgements}

We would like to acknowledge and thank those who have financially sponsored this work. This includes the Engineering and Physical Sciences Research Council (EPSRC) and Oxford University's centre for doctoral training in Industrially Focused Mathematical Modelling (InFoMM), with the EP/L015803/1 funding grant. Furthermore, this research stems from a PhD project \citep{sheridan2020nested} which was funded by Arm and NAG. Additionally, funding was also provided by the Inference, Computation and Numerics for Insights into Cities (ICONIC) project, and the programme grant EP/P020720/1. Lastly, Mansfield College Oxford also contributed funds.

\bibliography{../references}

\end{document}